\documentclass[a4paper,12pt]{amsart}
\usepackage{amssymb}
\usepackage{ifthen}
\usepackage{graphicx}
\usepackage{float}
\usepackage{caption}
\usepackage{subcaption}
\usepackage{cite}
\usepackage{amsfonts}
\usepackage{amscd}
\usepackage{amsxtra}
\usepackage{color}
\usepackage[dvipsnames]{xcolor}

\newtheorem{theorem}{Theorem}[section]
\newtheorem{lemma}[theorem]{Lemma}

\newtheorem{conj}{Conjecture}
\newtheorem{prob}{Problem}

\theoremstyle{definition}
\newtheorem{defn}{Definition}[section]
\newtheorem{example}{Example}

\newcounter{alphabet}


\newcommand{\Aut}{{\operatorname{Aut}}}

\def\be{\begin{equation}}
\def\ee{\end{equation}}

\newcommand{\bee}{\begin{enumerate}}
\newcommand{\eee}{\end{enumerate}}

\newcommand{\blem}{\begin{lem}}
\newcommand{\elem}{\end{lem}}
\newcommand{\bthm}{\begin{thm}}
\newcommand{\ethm}{\end{thm}}
\newcommand{\bcor}{\begin{cor}}
\newcommand{\ecor}{\end{cor}}
\newcommand{\beg}{\begin{example}}
\newcommand{\eeg}{\end{example}}
\newcommand{\begs}{\begin{examples}}
\newcommand{\eegs}{\end{examples}}
\newcommand{\bdefe}{\begin{defn}}
\newcommand{\edefe}{\end{defn}}
\newcommand{\bprob}{\begin{prob}}
\newcommand{\eprob}{\end{prob}}
\newcommand{\bques}{\begin{ques}}
\newcommand{\eques}{\end{ques}}
\newcommand{\bei}{\begin{itemize}}
\newcommand{\eei}{\end{itemize}}
\newcommand{\bcon}{\begin{conj}}
\newcommand{\econ}{\end{conj}}
\newcommand{\bcons}{\begin{conjs}}
\newcommand{\econs}{\end{conjs}}
\newcommand{\bprop}{\begin{propo}}
\newcommand{\eprop}{\end{propo}}
\newcommand{\br}{\begin{rem}}
\newcommand{\er}{\end{rem}}
\newcommand{\brs}{\begin{rems}}
\newcommand{\ers}{\end{rems}}
\newcommand{\bo}{\begin{obser}}
\newcommand{\eo}{\end{obser}}
\newcommand{\bos}{\begin{obsers}}
\newcommand{\eos}{\end{obsers}}
\newcommand{\bpf}{\begin{proof}}
\newcommand{\epf}{\end{proof}}
\newcommand{\ba}{\begin{array}}
\newcommand{\ea}{\end{array}}
\newcommand{\beq}{\begin{align}}
\newcommand{\beqq}{\begin{align*}}
\newcommand{\eeq}{\end{align}}
\newcommand{\eeqq}{\end{align*}}

\newcounter{minutes}
\divide\time by 60
\newcounter{hours}
\multiply\time by 60 \addtocounter{minutes}{-\time}

\begin{document}
\bibliographystyle{amsplain}

\title[]%
{The range of Hilbert operator and Derivative-Hilbert operator acting on $H^{\infty}$}

\author{Liyun Zhao Zhenyou Wang* and Zhirong Su}

\begin{abstract}
Let $\mu$ be a positive Borel measure on the interval $[0,1)$. The Hankel matrix
$\mathcal{H}_{\mu}=(\mu_{n,k})_{n,k\geq0}$
with entries $\mu_{n,k}=\mu_{n+k}$, where $\mu_n=\int_{[0,1)}t^{n}d\mu(t)$. For $f(z)=\sum_{n=0}^{\infty}a_nz^n$ is an analytic function in $\mathbb{D}$,
the Hilbert operator is defined by
$$\mathcal{H}_{\mu}(f)(z)=\sum_{n=0}^{\infty}\Bigg(\sum_{k=0}^{\infty}\mu_{n,k}a_k\Bigg)z^n,
\quad z\in \mathbb{D}.$$
The Derivative-Hilbert operator is defined as
$$\mathcal{DH}_{\mu}(f)(z)=\sum_{n=0}^{\infty}\Bigg(\sum_{k=0}^{\infty}\mu_{n,k}a_k\Bigg)(n+1)z^n,
\quad z\in \mathbb{D}.$$
In this paper, we determine the range of the Hilbert operator and Derivative-Hilbert operator acting on $H^{\infty}$.
\\\hspace*{\fill}\\
\textit{Keywords:}
Derivative-Hilbert operator, Hilbert operator, $H^{\infty}$, $\mathcal{Q}_p$ spaces, Carleson measure
\end{abstract}

\thanks{\ \ *Corresponding author. }
\thanks{\ \ \ \textit{Email addresses:}\texttt{1914407155@qq.com}(L.Y. Zhao),\ \ \texttt{zywang@gdut.edu.cn}(Z.Y. Wang),\\ \texttt{1131258896@qq.com}(Z.R. Su).\\
\textit{Addresses:Department of Mathematics and Statistics,Guangdong University of Technology, 510520 Guangzhou, Guangdong, P. R. China.}}

\maketitle
\section{Introduction}
Let $\mu$ be a positive Borel measure on the interval $[0,1)$. The Hankel matrix
$\mathcal{H}_{\mu}=(\mu_{n,k})_{n,k\geq0}$
with entries $\mu_{n,k}=\mu_{n+k}$, where $\mu_n=\int_{[0,1)}t^{n}d\mu(t)$.
For an analytic function $f(z)=\sum_{n=0}^{\infty}a_nz^n$, the generalized  Hilbert operator is defined as
$$\mathcal{H}_{\mu}(f)(z)=\sum_{n=0}^{\infty}\left(\sum_{k=0}^{\infty}\mu_{n,k}a_k\right)z^n,\quad z\in\mathbb{D},\eqno(1)$$
on the space of analytic functions in $\mathbb{D}$. The Derivative-Hilbert operator $\mathcal{DH}_{\mu}$ is first studied by Ye and Zhou \cite{author2,author3}, they defined $\mathcal{DH}_{\mu}$   as
$$\mathcal{DH}_{\mu}(f)(z)=\sum_{n=0}^{\infty}\left(\sum_{k=0}^{\infty}\mu_{n,k}a_k\right)(n+1)z^n,\quad \ z\in \mathbb{D},\eqno(2)$$
on the space of analytic functions in $\mathbb{D}$. If the terms on the right-hand sides of $(1),(2)$ make sense for all $z\in\mathbb{D}$, and the resulting functions are analytic in $\mathbb{D}$.  It is due to $$\mathcal{DH}_{\mu}(f)(z)=(z\mathcal{H}_{\mu}(f)(z))',$$
$\mathcal{DH}_{\mu}$ is called  Derivative-Hilbert operator. Another generalized integral operator related to $\mathcal{H}_{\mu}$ and $\mathcal{DH}_{\mu}$ (denoted by $\mathcal{I}_{\mu_{\alpha}},\,\alpha\in\mathbb{N}^{+}$) is defined by
$$\mathcal{I}_{\mu_{\alpha}}(f)(z)=\int_{[0,1)}\frac{f(t)}{(1-tz)^{\alpha}}d\mu(t).$$
When $\alpha=1$, we use $\mathcal{I}_{\mu}$  denote $\mathcal{I}_{\mu_1}$. In \cite{author18}, Galanopoulos and Pel\'{a}ez come to a conclusion that when $\mu$ is a Carleson measure on $[0,1)$, then $\mathcal{I}_{\mu}$ and $\mathcal{H}_{\mu}$ are well defined in $H^1$, moreover, $\mathcal{I}_{\mu}(f)=\mathcal{H}_{\mu}(f)$ for all $f\in H^1$.
In \cite{author1}, Chatzifountas extended Galanopoulos and Pel\'{a}ez's results to all Hardy space. He  characterized  measures $\mu$ for which $\mathcal{H}_{\mu}$ is bounded(compact) operator from $H^p$ into $H^q$, $0<p, q<\infty$.
Ye and Zhou characterized the measure $\mu$ for which $\mathcal{I}_{\mu_{2}}$ and $\mathcal{DH}_{\mu}$ is bounded\,(resp.,compact)\,on Bloch space in \cite{author2}. They did the similar researches on Bergman spaces in \cite{author3}. In \cite{author6}, Bao and Wulan gave another description about Carleson measure on $[0,1)$ and proved that when $0<p<2$, the range of the ces\`{a}ro-like operator acting on $H^{\infty}$ is a subset of $\mathcal{Q}_p$ if and only if $\mu$ is a Carleson measure.\\
\indent
Following the idea of the paper by Bao and Wulan \cite{author6}, In this paper, we determine the range of the Hilbert operator and Derivative-Hilbert operator acting on $H^{\infty}$.\\
\indent
Notation. Throughout this paper, $C$ denotes a positive constant which may be different from one occurrence to the next. The symbol $A\approx B$ means that $A\lesssim B\lesssim A$. We say that $A\lesssim B$ if there exists a positive constant $C$ such that $A\leq C B$.
\section{Notation and Preliminaries}

Let $\mathbb{D}=\{z:|z|\leq1\}$ and $\partial\mathbb{D}=\{z:|z|=1\}$ denote respectively the open unit disc and the unit circle in the complex plane $\mathbb{C}$. Let  $H(\mathbb{D})$ be the space of all analytic functions in $\mathbb{D}$ endowed with the topology of uniform convergence in compact subsets. \\
\indent
If $0<r<1$ and $f\in H(\mathbb{D})$, we set
\begin{align*}
M_p(r,f)&=\left(\frac{1}{2\pi}\int_{0}^{2\pi}|f(re^{it})|^pdt\right)^{\frac{1}{p}},\ 0<p<\infty,\\
M_{\infty}(r,f)&=\sup_{|z|=r}|f(z)|.
\end{align*}
For $0<p\leq \infty$, the Hardy space $H^{p}$ consists of those $f\in H(\mathbb{D})$ such that
$$\|H\|_{p}=\sup_{0<r<1}M_p(r,f)<\infty.$$
We refer to \cite{author7} for the notation and results about Hardy spaces.\\
\indent
The Bloch space $\mathcal{B}$  is the set of functions $f\in{H}(\mathbb{D})$ with $$\|f\|_{\mathcal{B}}=|f(0)|+\sup_{z\in\mathbb{D}}(1-|z|^2)|f{'}(z)|<\infty.$$ It is known that $\mathcal{B}$ is a Banach space with the norm $\|f\|_{\mathcal{B}}$. A classical reference for the theory of Bloch functions is \cite{author8}.\\
\indent
It is well known that the set of all disc automorphisms(i.e., of all one-to-one analytic maps of $\mathbb{D}$ onto itself), denoted Aut$(\mathbb{D})$, coincides with the set of all M\"{o}bius transformations of $\mathbb{D}$ onto itself: $$\Aut(\mathbb{D})=\{e^{i\theta}\sigma_a:a\in\mathbb{D}\  \mbox{and} \ \theta \  \mbox{is real}\},$$
where $$\sigma_a(z)=\frac{a-z}{1-\bar{a}z},\quad z\in\mathbb{D}.$$
\indent
A space $X$ of analytic functions in $\mathbb{D}$, defined via a semi-norm $\rho$, is said to be conformally invariant or M\"{o}bius invariant if whenever $f\in X$, then also $f\circ\phi\in X$  for any $\phi\in \mbox{Aut}(\mathbb{D})$ and moreover, $\rho(f\circ\phi)\leq C\rho(f)$ for some positive constant $C$ and all $f\in X$.

For $0\leq p<\infty,$ a function $f$ analytic in $\mathbb{D}$ belongs to $\mathcal{Q}_{p}$ if $$\|f\|_{\mathcal{Q}_p}^{2}=\sup_{w\in\mathbb{D}}\int_{\mathbb{D}}|f'(z)|^2(1-|\sigma_w(z)|^2)^pdA(z)<\infty.$$
Since $$\|f\circ\phi\|_{\mathcal{Q}_p}=\|f\|_{\mathcal{Q}_p}$$
for every $f\in\mathcal{Q}_p$ and $\phi\in \Aut(\mathbb{D})$, $\mathcal{Q}_p$ spaces are M\"{o}bius invariant spaces.
The space $\mathcal{Q}_0$ is the Dirichlet space $\mathcal{D}$ and the space $\mathcal{Q}_1$ is coincide with BMOA. When $0<p<1$, $\mathcal{Q}_p$ is a subset of BMOA, when $1<p<\infty$, $\mathcal{Q}_p=\mathcal{B}$. We refer to \cite{author9,author10} for the notation and results regarding $\mathcal{Q}_p$ spaces.\\
\indent
Let us start recalling the the mean Lipschitz space $\Lambda_{\alpha}^{p}$. For given $1\leq p\leq\infty$ and $0\leq \alpha\leq1$, the mean Lipschitz space $ \Lambda_{\alpha}^{p}$ consists of those functions $f$ analytic in $\mathbb{D}$ having a non-tangential limit almost everywhere for which $w_p(\delta,f)=\mathcal{O}(\delta^{\alpha}),\ \mbox{as}\ \delta \to 0$, where $w_p(.,f)$ is the integral modulus of continutiy of order $p$ of the boundary values $f(e^{i\theta})$ of $f$. A function $f\in H(\mathbb{D})$ belongs to $\Lambda_{\alpha}^{p}$, if
$$\|f\|_{p,\alpha}=|f(0)|+\sup_{0\leq r<1}(1-r)^{1-\alpha}M_p(r,f')<\infty.$$
A classical results about $\Lambda_{\alpha}^{p}$ is that $\Lambda_{\alpha}^{p}\subset H^p$ with $1\leq p\leq \infty$ and $0<\alpha\leq 1$(see \cite{author11}).
\\
\indent
For an arc ${I}\subseteq\partial\mathbb{D}$, let $|I|=\frac{1}{2\pi}\int_{I}|d\xi|$ be the normalized length of ${I}$ and ${S}({I})$ be the Carleson square based on $I$ with
$$ {S}({I})=\{z=re^{it}:e^{it}\in {I};1-|I|\leq r<1\}.$$
Clearly, if ${I}=\partial\mathbb{D}$, then ${S}({I})=\mathbb{D}$.\\
\indent
For $0<s<\infty$, we say that a positive Borel measure on $\mathbb{D}$ is a $s$-Carleson measure(See \cite{author12}) if
$$\sup_{I\subset\partial\mathbb{D}}\frac{\mu(S(I))}{|I|^s}<\infty.$$
If $s=1$, $s$-Carleson measure is the classical Carleson measure. When the positive Borel measure $\mu$ on $\mathbb{D}$ satisfies the following equation
$$\lim_{|I|\to 0}\frac{\mu(S(I))}{|I|^s}=0,$$
$\mu$ is a vanishing $s$-Carleson measure.
If $s=1$, the vanishing $s$-Carleson measure is the vanishing  Carleson measure.\\
 \indent
A positive Borel measure on $[0,1)$ also can be seen as a Borel measure on $\mathbb{D}$
by identifying it with the measure $\tilde{\mu}$ defined by
$$\tilde{\mu}(\textit E)=\mu(\textit E\cap[0,1)),$$
for any Borel subset $E$ of $\mathbb{D}$. Then a positive Borel measure $\mu$ on $[0,1)$
can be seen as an $s$-Carleson measure on $\mathbb{D}$, if
$$\sup_{t\in[0,1)} \frac{\mu([t,1))}{(1-t)^s}<\infty.$$
We have similar statement for vanishing $s$-Carleson measure.\\
\indent
Finally, we recall a general form of the Minkowski inequality which will be used in our main proof(see \cite{author21}, Appendices, A.1).\\
\indent Let $1\leq p<\infty$, then
$$\left[\int_{S_2}\left(\int_{S_1}|F(x,y)|d\mu_1(x)\right)^pd\mu_2(y)\right]^{\frac{1}{p}}\leq\int_{S_1}\left(\int_{S_2}|F(x,y)|^pd\mu_2(y)\right)^{\frac{1}{p}}d\mu_1(x).$$
Here $F(x,y)$ is a measurable function on the $\sigma$-finite product measure space $S_1\times S_2$; $d\mu_1(x)$ and $d\mu_2(y)$ are the measures on $S_1$ and $S_2$ respectively.
\\

\section{main results}
\indent A number of results will be needed to prove our main theorems. We start with a characterization of Carleson measure  on $[0,1)$ see  \cite{author6} for the detail process of proof.
\begin{lemma}\label{lem3.1}
Suppose $0<t<\infty$, $0\leq r<s<\infty$ and $\mu$ is a finite positive Borel measure on $[0,1)$. Then the following conditions are equivalent:\\
\\
\indent $(i)$ $\mu$ is an $s$-Carleson measure;\\
\\
\indent $(ii)$
$$\sup_{a\in\mathbb{D}}\int_{[0,1)}\frac{(1-|a|)^t}{(1-x)^r(1-|a|x)^{s+t-r}}d\mu(t)<\infty;$$\\
\indent $(iii)$
$$\sup_{a\in\mathbb{D}}\int_{[0,1)}\frac{(1-|a|)^t}{(1-x)^r|1-ax|^{s+t-r}}d\mu(t)<\infty.$$
\end{lemma}
We recall a charcterization of the functions $f\in H(\mathbb{D})$ whose Taylor coefficients is non-negative real number which belongs to $\mathcal{Q}_{p}$(Theorem $2.3$ in \cite{author13}).
\begin{lemma}\label{lem3.2}
Let $0<p<\infty$ and $f(z)=\sum_{n=0}^{\infty}a_nz^n$ be an analytic function in $\mathbb{D}$ with $a_n\geq 0$. Then $f\in\mathcal{Q}_{p}$ if and only if
$$\sup_{0\leq r<1}\sum_{n=0}^{\infty}\frac{(1-r)^p}{(n+1)^{p+1}}\left(\sum_{k=0}^{n}(k+1)a_{k+1}(n-k+1)^{p-1}r^{n-k}\right)^2<\infty.$$
\end{lemma}
The following lemma is from \cite{author14}.
\begin{lemma}\label{lem3.3}
Suppose $s>-1$, $r>0$, $t>0$ with $r+t-s-2>0$. If $r,t<2+s$, then
$$\int_{\mathbb{D}}\frac{(1-|z|^2)^s}{|1-\bar{a}z|^r|1-\bar{b}z|^t}dA(z)\lesssim\frac{1}{|1-\bar{a}b|^{r+t-s-2}}$$
for all $a,b\in\mathbb{D}$. If $t<2+s<r$, then
$$\int_{\mathbb{D}}\frac{(1-|z|^2)^s}{|1-\bar{a}z|^r|1-\bar{b}z|^t}dA(z)\lesssim\frac{(1-|a|^2)^{2+s-r}}{|1-\bar{a}b|^t}$$
for all $a,b\in\mathbb{D}$.
\end{lemma}
We shall also need the following Lemma $3.4$(see \cite{author15})which is a generalization of Lemma $3.1$ in \cite{author16} from $p=2$ to $1<p<\infty$.
\begin{lemma}\label{3.4}
Let $f\in H(\mathbb{D})$ with $f(z)=\sum_{n=0}^{\infty}a_nz^n$. Suppose $1<p<\infty$ and the sequence $\{a_n\}$ is a decreasing sequence of nonnegative numbers. If $X$ is a subsequence of $H(\mathbb{D})$ with $\Lambda_{{1}/{p}}^{p}\subseteq X\subseteq \mathcal{B}$, then $$f\in X\Longleftrightarrow a_n=\mathcal{O}\left(\frac{1}{n}\right).$$
\end{lemma}

Finally we recall the following result which is another  characterization of $s$-Carleson measure $\mu$ on $[0,1)$(see \cite{author1}, Proposition $1$).
\begin{lemma}\label{3.5}
Let $\mu$ be a positive finite Borel measure on $[0,1)$ and $s>0$. Then $\mu$ is a $s$-Carleson measure if and only if the sequence of moments $\{\mu_n\}_{n=0}^{\infty}$ satisfies $$\sup_{n\geq0}(1+n)^s\mu_n<\infty.$$
\end{lemma}

\begin{theorem}\label{3.6}
Suppose $0<p<2$ and $\mu$ is a positive finite Borel measure on $[0,1)$. Then $\mathcal{H}_{\mu}(H^{\infty})\subseteq\mathcal{Q}_{p}$ if and only if $\mu$ is a Carleson measure.
\end{theorem}
\begin{proof}
Let $\mathcal{H}_{\mu}(H^{\infty})\subseteq\mathcal{Q}_{p}$. Take $f(z)=1\in H^{\infty}$, Then $$\mathcal{H}_{\mu}(f)(z)=\sum_{n=0}^{\infty}\mu_nz^n\in\mathcal{Q}_{p}.$$
  Bearing in mind that the mean Lipschitz space $\Lambda_{{1}/{2}}^{2}$ is contained in all the $\mathcal{Q}_p$ spaces(See \cite{author20}, Remark $4$, p.427), we get that $\Lambda_{{1}/{2}}^{2}\subseteq\mathcal{Q}_{p}\subseteq {\mathcal{B}}$. Using Lemma $3.4$, we  imply  $\mu_n=\mathcal{O}\left(\frac{1}{n}\right)$, then  Lemma $3.5$ gives that $\mu$ is a Carleson measure.\\
 \indent On the other hand, let $\mu$ be a Carleson measure and $f\in H^{\infty}\subseteq H^1$. The Proposition $1.1$ in \cite{author18} gives that
 $$\mathcal{H}_{\mu}(f)(z)=\int_{[0,1)}\frac{f(t)}{1-tz}d\mu(t),\quad f\in H^1, \quad z\in\mathbb{D}.$$
  Hence for any $z\in \mathbb{D}$,
  \begin{eqnarray*}
&&\lefteqn{\|\mathcal{H}_{\mu}(f)\|_{\mathcal{Q}_p}}\\
&&=\sup_{a\in\mathbb{D}}\left(\int_{\mathbb{D}}\left|\int_{[0,1)}\frac{tf(t)}{(1-tz)^2}d\mu(t)\right|^2(1-|\sigma_a(z)|^2)^{p}dA(z)\right)^{\frac{1}{2}}\\
&&\lesssim\|f\|_{H^{\infty}}\sup_{a\in\mathbb{D}}\left(\int_{\mathbb{D}}\left(\int_{[0,1)}\frac{1}{|1-tz|^2}d\mu(t)\right)^2(1-|\sigma_a(z)|^2)^p dA(z)\right)^{\frac{1}{2}}.\\
\end{eqnarray*}
By the  Minkowski inequality, Lemma $3.3$  and Lemma $3.1$, we get
 \begin{eqnarray*}
&&\lefteqn{\sup_{a\in\mathbb{D}}\left(\int_{\mathbb{D}}\left(\int_{[0,1)}\frac{1}{|1-tz|^2}d\mu(t)\right)^2(1-|\sigma_a(z)|^2)^pdA(z)\right)^{\frac{1}{2}}}\\
&&\leq\sup_{a\in\mathbb{D}}\int_{[0,1)}\left(\int_{\mathbb{D}}\frac{1}{|1-tz|^4}(1-|\sigma_a(z)|^2)^pdA(z)\right)^{\frac{1}{2}}d\mu(t)\\
&&\lesssim\sup_{a\in\mathbb{D}}\int_{[0,1)}(1-|a|^2)^{\frac{p}{2}}\left(\int_{\mathbb{D}}\frac{(1-|z|^2)^p}{|1-tz|^4|1-\bar{a}z|^{2p}}dA(z)\right)^{\frac{1}{2}}d\mu(t)\\
&&\lesssim\int_{[0,1)}\frac{(1-|a|^2)^{\frac{p}{2}}}{(1-t^2)^{1-\frac{p}{2}}|1-ta|^p}d\mu(t)<\infty.\\
\end{eqnarray*}
 We obtain that $\mathcal{H}_{\mu}(f)\subseteq\mathcal{Q}_{p}$. The proof is complete.
\end{proof}

\begin{lemma}\label{3.7}
Suppose $p>0$  and let $\mu$ be a positive Borel measure on $[0,1)$. Then for any given $f\in \mathcal{Q}_{p}$ and $\alpha\in \mathbb{N}^+$, the intergral $$\mathcal{I}_{\mu_{\alpha}}(f)(z)=\int_{[0,1)}\frac{f(t)}{(1-tz)^{\alpha}}d\mu(t)$$ uniformly converges on any compact subset of $\mathbb{D}$ if and only if the measure $\mu$ satisfies $\int_{[0,1)}\log\frac{2}{1-t}d\mu(t)<\infty$.
\end{lemma}
\begin{proof}
The proof is similar to the proof of Theorem $2.1$ in \cite{author2}.\\
Suppose that $\mu$ satisfies $M=\int_{[0,1)}\log\frac{2}{1-t}d\mu(t)<\infty$.  It is due to the fact that any $f\in\mathcal{Q}_p$ has the growth  $$|f(z)|\lesssim\|f\|_{\mathcal{Q}_p}\log\frac{2}{1-|z|^2},\quad f \in\mathcal{Q}_p, \quad z\in \mathbb{D}.$$
Then for any $f\in\mathcal{Q}_p$,\ $\alpha\in\mathbb{N}^+$,\ $0<r<1$ and all $z$ with $|z|\leq r$, we obtain
\begin{align*}
\int_{[0,1)}\frac{|f(t)|}{|1-tz|^{\alpha}}d\mu(t)&<\frac{1}{(1-r)^{\alpha}}\int_{[0,1)}|f(t)|d\mu(t)\\
&\lesssim\|f\|_{\mathcal{Q}_p}\frac{1}{(1-r)^{\alpha}}\int_{[0,1)}\log\frac{2}{1-t^2}d\mu(t)\\
&\lesssim\|f\|_{\mathcal{Q}_p}\frac{1}{(1-r)^{\alpha}}\int_{[0,1)}\log\frac{2}{1-t}d\mu(t)\\
&=\frac{M\|f\|_{\mathcal{Q}_p}}{(1-r)^{\alpha}}.
\end{align*}
Hence the intergral $$\mathcal{I}_{\mu_{\alpha}}(f)(z)=\int_{[0,1)}\frac{f(t)}{(1-tz)^{\alpha}}d\mu(t)$$ uniformly converges on any compact subset of $\mathbb{D}$ and the resulting funtion $\mathcal{I}_{\mu_{\alpha}}$ is analytic in $\mathbb{D}$.\\
\indent
Conversely suppose that the operator $\mathcal{I}_{\mu_{\alpha}}$ is well defined in $\mathcal{Q}_{p}$. Let $f(z)=\log\frac{2}{1-t}$. It is well know that $f\in \mathcal{Q}_{p}$. Then it follows that $\mathcal{I}_{\mu_{\alpha}}(f)(z)$ is well defined for every $z\in\mathbb{D}$. In particular,
$$\mathcal{I}_{\mu_{\alpha}}(f)(0)=\int_{[0,1)}\log\frac{2}{1-t}d\mu(t)$$
is a complex number. Since $\mu$ is a positive measure and $\log\frac{2}{1-t}>0$ for all $t\in[0,1)$, we get that
$$\int_{[0,1)}\log\frac{2}{1-t}d\mu(t)<\infty.$$
The proof is complete.

\end{proof}

The following lemma is a characterization of the coefficient multipliers from $\mathcal{B}$ into $l^1$(see \cite{author19}).

\begin{lemma}\label{3.8}
A sequence $\{\lambda_n\}_{n=0}^{\infty}$ of complex number is a coefficient multiplier from $\mathcal{B}$ into $l^1$ if and only if $$\sum_{n=1}^{\infty}\left(\sum_{k=2^n+1}^{2^{n+1}}|\lambda_k|^2\right)^{\frac{1}{2}}<\infty.$$
\end{lemma}

\begin{theorem}\label{3.9}
Suppose $0<p<2$ and $\mu$ is a positive finite Borel measure on $[0,1)$ which satisfies $\int_{[0,1)}\log\frac{2}{1-t}d\mu(t)<\infty$. Then $\mathcal{DH}_{\mu}(H^{\infty})\subseteq\mathcal{Q}_{p}$ if and only if $\mu$ is a $2$-Carleson measure.
\end{theorem}

\begin{proof}
Suppose $\mathcal{DH}_{\mu}(H^{\infty})\subseteq\mathcal{Q}_{p}$. Take $f(z)=1\in H^{\infty}$, then
$$\mathcal{DH}_{\mu}(f)(z)=\sum_{n=0}^{\infty}(n+1)\mu_{n}z^n\in\mathcal{Q}_{p}.$$
Using Lemma $3.2$, we deduce
\begin{align*}
\infty&>\sum_{n=0}^{\infty}\frac{(1-r)^p}{(n+1)^{p+1}}\left(\sum_{k=0}^{\infty}(k+2)^2\mu_{k+1}(n-k+1)^{p-1}r^{n-k}\right)^2\\
&\gtrsim\sum_{n=0}^{\infty}\frac{(1-r)^p}{(4n+1)^{p+1}}\left(\sum_{k=0}^{4n}(k+2)^2\mu_{k+1}(4n-k+1)^{p-1}r^{4n-k}\right)^2\\
&\gtrsim\sum_{n=0}^{\infty}\frac{(1-r)^p}{(4n+1)^{p+1}}\left(\sum_{k=n}^{2n}(k+2)^2\int_{r}^{1}t^{k+1}d\mu(t)(4n-k+1)^{p-1}r^{4n-k}\right)^2\\
&\gtrsim\mu^{2}([r,1))(1-r)^p\sum_{n=0}^{\infty}\frac{r^{8n+2}}{(4n+1)^{p+1}}\left(\sum_{k=n}^{2n}(k+2)^2(4n-k+1)^{p-1}\right)^2\\
&\gtrsim\mu^{2}([r,1))(1-r)^{p}\sum_{n=0}^{\infty}(4n+2)^{4+p-1}r^{8n+2}\\
&\approx\frac{\mu^{2}([r,1))}{(1-r)^4}
\end{align*}
for all $r\in[0,1)$ which yields that $\mu$ is a $2$-Carleson measure.\\
\indent Conversely suppose $\mu$ is a $2$-Carleson measure and $f=\sum_{k=0}^{\infty}a_kz^k\in H^{\infty}\in\mathcal{Q}_p$. Since $\int_{[0,1)}\log\frac{2}{1-t}d\mu(t)<\infty$, we obtain the integral$\int_{[0,1)}t^nf(t)d\mu(t)$ converges absolutely and
$$\sup_{n\geq0}\left|\int_{[0,1)}t^nf(t)d\mu(t)\right|\leq\|f\|_{\mathcal{Q}_p}\int_{[0,1)}\log\frac{2}{1-t}d\mu(t)<\infty.$$
It follows from Lemma $3.7$ that $\int_{[0,1)}\frac{f(t)}{(1-tz)^2}d\mu(t)$ converges absolutely, then we obtain
\begin{align*}
\mathcal{I}_{\mu_2}(f)(z)&=\int_{[0,1)}\frac{f(t)}{(1-tz)^2}d\mu(t)\\
&=\int_{[0,1)}f(t)\left(\sum_{n=0}^{\infty}(n+1)t^nz^n\right)d\mu(t)\\
&=\sum_{n=0}^{\infty}(n+1)\left(\int_{[0,1)}t^nf(t)d\mu(t)\right)z^n.
\end{align*}
 Using Lemma $3.5$, we see that there exists $C>0$ such that
$$|\mu_n|\leq\frac{C}{n^2},\quad n>0.$$
Then it follows that
$$\sum_{n=1}^{\infty}\left(\sum_{k=2^n+1}^{2^{n+1}}|\mu_k|^2\right)^{\frac{1}{2}}\lesssim \sum_{n=1}^{\infty}\left(\sum_{k=2^n+1}^{2^{n+1}}\frac{1}{k^4}\right)^{\frac{1}{2}}\lesssim\sum_{n=1}^{\infty}\frac{1}{2^{\frac{3n+4}{2}}}<\infty$$
these together with Lemma $3.8$ and $\mathcal{Q}_p\subseteq \mathcal{B}$, we obtain that the sequence of moments $\{\mu_n\}_{n=0}^{\infty}$ is a multiplier from $\mathcal{Q}_p$ to $l^1$. Then, there exists a constant $C>0$ such that
$$\sum_{k=0}^{\infty}|\mu_{n,k}a_k|\leq\sum_{k=0}^{\infty}|\mu_{k}a_k|\leq C\|f\|_{\mathcal{Q}_p},$$
which implies that $\mathcal{DH}_{\mu}(f)$ is a well defined function in $\mathbb{D}$ and
$$\sum_{k=0}^{\infty}\mu_{n,k}a_k=\sum_{k=0}^{\infty}a_k\int_{[0,1)}t^{n+k}d\mu(t)=\int_{[0,1)}t^nf(t)d\mu(t).$$
It follows that
\begin{align*}
\mathcal{I}_{\mu_2}(f)&=\sum_{n=0}^{\infty}(n+1)\left(\int_{[0,1)}t^nf(t)d\mu(t)\right)z^n\\
&=\sum_{n=0}^{\infty}\left(\sum_{k=0}^{\infty}\mu_{n,k}a_k\right)(n+1)z^n\\
&=\mathcal{DH}_{\mu}(f).
\end{align*}

Since $$\mathcal{DH}_{\mu}(f)(z)=\int_{[0,1)}\frac{f(t)}{(1-tz)^2}d\mu(t),\quad z\in \mathbb{D}.$$
Hence for any $z\in \mathbb{D}$,
\begin{eqnarray*}
&&\lefteqn{\|\mathcal{DH}_{\mu}(f)\|_{\mathcal{Q}_p}}\\
&&=\sup_{a\in\mathbb{D}}\left(\int_{\mathbb{D}}\left|\int_{[0,1)}\frac{2tf(t)}{(1-tz)^3}d\mu(t)\right|^2(1-|\sigma_a(z)|^2)^{p}dA(z)\right)^{\frac{1}{2}}\\
&&\lesssim\|f\|_{H^{\infty}}\sup_{a\in\mathbb{D}}\left(\int_{\mathbb{D}}\left(\int_{[0,1)}\frac{1}{|1-tz|^3}d\mu(t)\right)^2(1-|\sigma_a(z)|^2)^p dA(z)\right)^{\frac{1}{2}}.\\
\end{eqnarray*}
It follows from the Minkowski inequality, Lemma $3.3$  and Lemma $3.1$ that
\begin{eqnarray*}
&&\lefteqn{\sup_{a\in\mathbb{D}}\left(\int_{\mathbb{D}}\left(\int_{[0,1)}\frac{1}{|1-tz|^3}d\mu(t)\right)^2(1-|\sigma_a(z)|^2)^pdA(z)\right)^{\frac{1}{2}}}\\
&&\leq\sup_{a\in\mathbb{D}}\int_{[0,1)}\left(\int_{\mathbb{D}}\frac{1}{|1-tz|^6}(1-|\sigma_a(z)|^2)^pdA(z)\right)^{\frac{1}{2}}d\mu(t)\\
&&\lesssim(1-|a|^2)^{\frac{p}{2}}\sup_{a\in\mathbb{D}}\int_{[0,1)}\left(\int_{\mathbb{D}}\frac{(1-|z|^2)^p}{|1-tz|^6|1-\bar{a}z|^{2p}}dA(z)\right)^{\frac{1}{2}}d\mu(t)\\
&&\lesssim\int_{[0,1)}\frac{(1-|a|^2)^{\frac{p}{2}}}{(1-t^2)^{2-\frac{p}{2}}|1-ta|^p}d\mu(t)<\infty.\\
\end{eqnarray*}
From the above proof, we get that $\mathcal{DH}_{\mu}(H^{\infty})\subseteq\mathcal{Q}_{p}$. The proof is complete.
\end{proof}

\section*{Acknowledgments }
The authors thank the referee for useful remarks and comments that led to the improvement of this paper.
The first author was partially supported by Guangdong Basic and Applied Basic Research Foundation (No. 2020B1515310001).

\end{document}